\documentclass[12pt,reqno]{amsart}
\usepackage{amsmath, amsthm, amscd, amsfonts, amssymb, graphicx, color}
\usepackage[bookmarksnumbered, colorlinks, plainpages]{hyperref}
\usepackage[cp1251]{inputenc}
\usepackage[T2A]{fontenc}
\usepackage[english]{babel}
\usepackage{amsmath,amsfonts,amssymb}
\usepackage{geometry}
\usepackage{cite}

\begin{document}

\title[Continuity in a parameter of solutions to generic problems]{Continuity in a parameter of solutions to generic boundary-value problems}

\author[V. Mikhailets, A. Murach, V. Soldatov]{Vladimir A. Mikhailets, Aleksandr A. Murach, Vitalii Soldatov}

\address{Vladimir A. Mikhailets \newline
Department of Nonlinear Analysis,
Institute of Mathematics, National Academy of Sciences of Ukraine,
Tereshchenkivska Str. 3, 01004 Kyiv-4, Ukraine}
\email{mikhailets@imath.kiev.ua}

\address{Aleksandr A. Murach \newline
Department of Nonlinear Analysis,
Institute of Mathematics, National Academy of Sciences of Ukraine,
Tereshchenkivska Str. 3, 01004 Kyiv-4, Ukraine;\newline
Chernihiv National T.G. Shevchenko Pedagogical University, Chernihiv, Ukraine}
\email{murach@imath.kiev.ua}

\address{Vitalii Soldatov \newline
Department of Nonlinear Analysis,
Institute of Mathematics, National Academy of Sciences of Ukraine,
Tereshchenkivska Str. 3, 01004 Kyiv-4, Ukraine}
\email{soldatovvo@ukr.net}


\begin{abstract}
We introduce the most general class of linear boundary-value problems for systems of first-order ordinary differential equations whose solutions belong to the complex H\"older space $C^{n+1,\alpha}$, with $0\leq n\in\mathbb{Z}$ and $0\leq\alpha\leq1$. The boundary conditions can contain derivatives $y^{(r)}$, with $1\leq r\leq n+1$, of the solution $y$ to the system. For parameter-dependent problems from this class, we obtain constructive criterion under which their solutions are continuous in the normed space $C^{n+1,\alpha}$ with respect to the parameter.
\end{abstract}

\maketitle
\numberwithin{equation}{section}
\newtheorem{theorem}{Theorem}[section]
\newtheorem{lemma}[theorem]{Lemma}
\newtheorem{proposition}[theorem]{Proposition}
\newtheorem{remark}[theorem]{Remark}
\newtheorem{example}[theorem]{Example}
\newtheorem{corollary}[theorem]{Corollary}
\allowdisplaybreaks

\renewcommand{\baselinestretch}{1.1}

\section{Introduction}\label{1}

Questions concerning the validity of passage to the limit in parameter-dependent differential equations arise in various problems. These questions are best clear up for the Cauchy problem for systems of first-order ordinary differential equations. Gikhman \cite{Gikhman1952}, Krasnosel'skii and S.~Krein \cite{KrasnoselskiiKrein1955}, Kurzweil and Vorel \cite{KurzweilVorel1957CMJ} obtained fundamental results relating to continuous dependence in the parameter of solutions to the Cauchy problem for nonlinear systems. For linear systems, these results were refined and supplemented by Levin \cite{Levin1967dan}, Opial \cite{Opial1967}, Reid \cite{Reid1967}, and Nguyen The Hoan \cite{NguenTheHoan1993}.

Parameter-dependent boundary-value problems are far less studied than the Cauchy problem. Kiguradze  \cite{Kiguradze1987, Kiguradze1975, Kiguradze2003} and Ashordia \cite{Ashordia1996} introduced and investigated the class of general linear boundary-value problems for systems of first-order ordinary differential equations. The solutions $y$ to these problems are supposed to be absolutely continuous on a compact interval $[a,b]$, and the boundary condition is of the form $By=q$ where $B$ is an arbitrary linear continuous operator from $C([a,b],\mathbb{R}^{m})$ to $\mathbb{R}^{m}$ ($m$ is the number of differential equations in the system). Kiguradze and Ashordia obtained conditions under which the solutions to a parameter-dependent problem from this class are continuous in $C([a,b],\mathbb{R}^{m})$ with respect to the parameter. Recently these results were refined and generalized to complex-valued functions and systems of higher-order differential equations \cite{MikhailetsReva2008DAN9, KodliukMikhailetsReva2013, MikhailetsChekhanova2015}.

In this paper we introduce a new class of boundary-value problems for systems of first-order linear differential equations. In contrast to the usual boundary-value problems, this class relates to a given function space. We consider systems whose coefficients and right-hand sides belong to the H\"older space $(C^{n,\alpha})^{m}:=C^{n,\alpha}([a,b],\mathbb{C}^{m})$, with $0\leq n\in\mathbb{Z}$ and $0\leq\alpha\leq1$. Since the solutions $y$ to each of these systems run through the whole space $(C^{n+1,\alpha})^{m}$, we consider the most general boundary condition of the form $By=q$ where $B$ is an arbitrary linear continuous operator from $(C^{n+1,\alpha})^{m}$ to $\mathbb{C}^{m}$. This condition can contain the derivatives $y^{(r)}$, with $1\leq r\leq n$, of the solution. We say that these boundary-value problems are generic with respect to the H\"older space $C^{n+1,\alpha}$.

We investigate parameter-dependent boundary-value problems from the class introduced. We will find sufficient and necessary constructive conditions under which the solutions to these problems are continuous in $(C^{n+1,\alpha})^{m}$ with respect to the parameter. We will also prove a two-sided estimate for the degree of convergence of the solutions.

Note that the generic boundary-value problems with respect to the Sobolev spaces and to the spaces $C^{(n+1)}$ were introduced in
\cite{MikhailetsReva2008DAN8, KodliukMikhailets2013JMS} and \cite{MikhailetsChekhanova2014DAN7, Soldatov2015UMJ}, where sufficient conditions for continuous dependence in parameter of solutions were obtained. These results were applied to the investigation of multipoint boundary-value problems \cite{Kodliuk2012Dop11}, Green's matrices of boundary-value problems \cite{KodliukMikhailetsReva2013, MikhailetsChekhanova2015}, in the spectral theory of differential operators with distributional coefficients \cite{GoriunovMikhailets2010MN87.2, GoriunovMikhailetsPankrashkin2013EJDE, GoriunovMikhailets2012UMJ63.9}.

The approach developed in this paper can be applied to generic boundary-value problems with respect to other function spaces.

\section{Main results}\label{3sec2}

Throughout the paper we arbitrarily fix a compact interval $[a,b]\subset\mathbb{R}$, integers $n\geq0$ and $m\geq1$, and a real number $\alpha$ such that $0\leq\alpha\leq1$. We use the H\"older spaces
\begin{equation}\label{3.Holder}
(C^{l,\alpha})^{m}:=C^{l,\alpha}([a,b],\mathbb{C}^{m})\quad\mbox{and}
\quad(C^{l,\alpha})^{m\times m}:=C^{l,\alpha}([a,b],\mathbb{C}^{m\times m}),
\end{equation}
with $0\leq l\in\mathbb{Z}$. They consist respectively of all vector-valued functions and $(m\times m)$-matrix-valued functions whose components belong to $C^{l,\alpha}:=C^{l,\alpha}([a,b],\mathbb{C})$, with the norm of the functions being equal to the sum of the norms in $C^{l,\alpha}$ of all their components. We denote all these norms by $\|\cdot\|_{l,\alpha}$. Certainly, if $\alpha=0$, then $C^{l,\alpha}$ stands for the space $C^{(l)}:=C^{(l)}([a,b],\mathbb{C})$ of all $l$ times continuously differentiable functions on $[a,b]$. We will recall the definition of the H\"older space $C^{l,\alpha}$ at the end of this section.

Let the parameter $\varepsilon$ runs through the set $[0,\epsilon_0)$, with the number $\epsilon_0>0$ being fixed. We consider the family of the following linear boundary value-problems:
\begin{gather}\label{3.syste}
L(\epsilon)y(t,\epsilon)\equiv y'(t,\epsilon)+
A(t,\epsilon)y(t,\epsilon)=f(t,\epsilon),\quad a\leq t\leq b,\\
B(\epsilon)y(\cdot,\epsilon)=q(\epsilon).\label{3.kue}
\end{gather}
Here, for every $\epsilon\in[0,\epsilon_0)$, we suppose that $y(\cdot,\epsilon)\in(C^{n+1,\alpha})^{m}$ is an unknown vector-valued
function, whereas the matrix-valued function
$A(\cdot,\epsilon)\in(C^{n,\alpha})^{m\times m}$, vector-valued function $f(\cdot,\epsilon)\in(C^{n,\alpha})^m$, continuous linear operator
\begin{equation}\label{3.Be}
B(\epsilon):(C^{n+1,\alpha})^m \to \mathbb{C}^{m}
\end{equation}
and vector $q(\epsilon)\in \mathbb{C}^{m}$ are arbitrarily given. Note that we interpret vectors as columns.

The boundary condition \eqref{3.kue} with the continuous operator \eqref{3.Be} is the most general for equation~\eqref{3.syste} in view of Lemma~\ref{3.sk}, which will be given in the next section. This condition covers all the classical types of boundary conditions such as initial conditions in the Cauchy problem, various multipoint conditions, integral conditions, conditions used in mixed boundary-value problems, and also nonclassical conditions containing the derivatives $y^{(r)}(\cdot,\varepsilon)$ with $1\leq r\leq n+1$.

By analogy with papers \cite{KodliukMikhailets2013JMS}  and \cite{Soldatov2015UMJ} we say that the boundary-value problem \eqref{3.syste}, \eqref{3.kue} is generic with respect to the space $C^{n+1,\alpha}$. (In these papers, the notion of a generic (or, in other words, total) boundary-value problem is introduced with respect to the Sobolev spaces and spaces of continuously differentiable functions respectively).

For the boundary-value problem \eqref{3.syste}, \eqref{3.kue}, we consider the following four

\medskip

\noindent\textbf{Limit Conditions} as $\epsilon\to0+$:
\begin{itemize}
  \item [(I)] $A(\cdot,\epsilon)\to A(\cdot,0)$ in $(C^{n,\alpha})^{m\times m}$;
  \item [(II)] $B(\epsilon)y\to B(0)y$ in $\mathbb{C}^{m}$ for every $y\in(C^{n+1,\alpha})^m$;
  \item [(III)] $f(\cdot,\epsilon)\to f(\cdot,0)$ in $(C^{n,\alpha})^{m}$;
  \item [(IV)] $q(\epsilon)\to q(0)$ in $\mathbb{C}^{m}$.
\end{itemize}

\smallskip

We also consider

\medskip

\noindent\textbf{Condition (0).} The limiting homogeneous boundary-value problem
\begin{gather*}
L(0)\,y(t,0)=0,\quad a\leq t\leq b,\quad\mbox{and}\quad B(0)\,y(\cdot,0)=0
\end{gather*}
has only the trivial solution.

\medskip

Let us formulate our

\medskip

\noindent\textbf{Basic Definition.} We say that the solution to the boundary-value problem \eqref{3.syste}, \eqref{3.kue} is continuous in the parameter $\varepsilon$ at $\varepsilon=0$ if the following two conditions are fulfilled:
\begin{itemize}
\item [$(\ast)$] There exists a positive number $\varepsilon_{1}<\varepsilon_{0}$ such that for arbitrary $\varepsilon\in[0,\varepsilon_{1})$, $f(\cdot,\epsilon)\in(C^{n,\alpha})^m$, and $q(\varepsilon)\in\mathbb{C}^{m}$ this problem has a unique solution $y(\cdot,\epsilon)\in(C^{n+1,\alpha})^{m}$.
      \item [$(\ast\ast)$] Limit Conditions (III) and (IV) imply that
\begin{equation}\label{3.gu}
y(\cdot,\epsilon)\to y(\cdot,0)\quad\mbox{in}\quad(C^{n+1,\alpha})^{m}
\quad\mbox{as}\quad\epsilon\to0+.
\end{equation}
\end{itemize}

\smallskip

\noindent\textbf{Main Theorem.} \it The solution to the boundary-value problem \eqref{3.syste}, \eqref{3.kue} is continuous in the parameter $\varepsilon$ at $\varepsilon=0$ if and only if this problem satisfies Condition~\rm(0) \it and Limit Conditions \rm (I) \it and \rm (II)\it.\rm

\medskip

We supplement Main Theorem by

\begin{theorem}\label{3.th-bound}
Assume that the boundary-value problem \eqref{3.syste}, \eqref{3.kue} satisfies Condition~\rm(0) \it and Limit Conditions \rm (I) \it and \rm (II)\it. Then there exist positive numbers $\epsilon_{2}<\epsilon_{1}$ and $\varkappa_{1}$, $\varkappa_{2}$ such that for every $\epsilon\in(0,\epsilon_{2})$ we have the two-sided estimate
\begin{equation}\label{3.bound}
\begin{aligned}
\varkappa_{1}
&\bigl(\,\|L(\epsilon)y(\cdot,0)-f(\cdot,\epsilon))\|_{n,\alpha}+
\|B(\epsilon)y(\cdot,0)-q(\epsilon)\|_{\mathbb{C}^{m}}\bigr)\\
&\leq\|y(\cdot,0)-y(\cdot,\epsilon)\|_{n+1,\alpha}\\
&\leq\varkappa_{2}
\bigl(\,\|L(\epsilon)y(\cdot,0)-f(\cdot,\epsilon))\|_{n,\alpha}+
\|B(\epsilon)y(\cdot,0)-q(\epsilon)\|_{\mathbb{C}^{m}}\bigr).
\end{aligned}
\end{equation}
Here, the numbers $\epsilon_{2}$, $\varkappa_{1}$, and $\varkappa_{2}$ do not depend on $y(\cdot,0)$, $y(\cdot,\epsilon)$, $f(\cdot,\epsilon)$, and $q(\epsilon)$. \rm
\end{theorem}

According to \eqref{3.bound}, the error and discrepancy of the solution $y(\cdot,\epsilon)$ to the boundary-value problem \eqref{3.syste}, \eqref{3.kue} are of the same degree. Here, we consider $y(\cdot,0)$ as an approximate solution to this problem.

We will prove Main Theorem and Theorem~\ref{3.th-bound} in Section~\ref{3sec4}.

\begin{remark}\rm
Let us discuss the boundary-value problem \eqref{3.syste}, \eqref{3.kue} and Main Theorem in the important case where $\alpha=0$. In this case an arbitrary linear continuous operator \eqref{3.Be} can be uniquely represented in the form
\begin{equation}\label{Be-description}
B(\varepsilon)z=
\sum\limits_{k=1}^{n+1}\beta_{k}(\varepsilon)\,z^{(k-1)}(a)+
\int\limits_a^b(d\Phi(t,\varepsilon))z^{(n+1)}(t),\quad z\in(C^{(n+1)})^m,
\end{equation}
in which every $\beta_{k}(\varepsilon)$ is a number $m\times m$-matrix whereas $\Phi(\cdot,\varepsilon)$ is an $m\times m$-matrix-valued function formed by scalar functions that are of bounded variation on $[a,b]$, right-continuous on $(a,b)$, and equal to zero at $t=a$. Here, the integral is understood in the Riemann-Stieltjes sense. This representation follows from the known description of the dual of  $C^{(n+1)}$; see, e.g., \cite[p.~344]{DanfordShvarts1958}. Applying  \eqref{Be-description}, we can reformulate Limit Condition~(II) in an explicit form. Namely, Limit Condition~(II) is equivalent to that the following four conditions are fulfilled as $\varepsilon\to0+$:
\begin{itemize}
  \item [(2a)] $\beta_{k}(\varepsilon)\to\beta_{k}(0)$ for every $k\in\{1,\ldots,n+1\}$;
  \item [(2b)] $\|V_{a}^{b}\Phi(\cdot,\epsilon)\|_{\mathbb{C}^{m\times m}}
      =O(1)$;
  \item [(2c)] $\Phi(b,\epsilon)\to\Phi(b,0)$;
  \item [(2d)] $\int_{a}^{t}\Phi(s,\epsilon)ds\to
\int_{a}^{t}\Phi(s,0)ds$ for every $t\in(a,b]$.
\end{itemize}
(Here, of course, the convergence is considered in $\mathbb{C}^{m\times m}$.) This equivalence follows from the Riesz theorem on the weak convergence of linear continuous functionals on $C^{(0)}([a,b],\mathbb{C})$ (see, e.g., \cite[Ch.~III, Sect.~55]{RieszSz-Nagy56}). It is useful to compare these conditions with the criterion of the norm convergence of the operators $B(\varepsilon)\to B(0)$ as $\varepsilon\to0+$. According to this criterion, the latter convergence is equivalent to that condition (2a) and
\begin{equation}\label{3.conv-in-variation}
V_{a}^{b}\bigl(\Phi(\cdot,\epsilon)-\Phi(\cdot,0)\bigr)\to0
\end{equation}
are fulfilled as $\varepsilon\to0+$. Condition \eqref{3.conv-in-variation} is much stronger than the system of conditions (2b)--(2d). This becomes especially clear if we observe that condition~\eqref{3.conv-in-variation} implies the uniform convergence of the functions $\Phi(t,\varepsilon)$ to $\Phi(t,0)$ on $[a,b]$ as $\varepsilon\to0+$ whereas conditions (2b)--(2d) do not imply the pointwise convergence of these functions at least in one point of $(a,b)$.

\end{remark}
At the end of this section, we recall the definition of the H\"older spaces and discuss some notions and designations relating to these spaces. Let an integer $l\geq0$. We endow the Banach space $C^{(l)}:=C^{(l)}([a,b],\mathbb{C})$ of all $l$ times continuously differentiable functions $x:[a,b]\rightarrow\mathbb{C}$ with the norm
$$
\|x\|_{l}:=\sum_{j=0}^{l}
\max\bigl\{|x^{(j)}(t)|:\,t\in[a,b]\bigr\}.
$$
By definition, the H\"older space $C^{l,\alpha}:=C^{l,\alpha}([a,b],\mathbb{C})$, with $0<\alpha\leq1$, consists of all functions $x\in\nobreak C^{(l)}$ such that
$$
\|x^{(l)}\|_{\alpha}':=
\sup\Bigl\{\frac{|x^{(l)}(t_2)-x^{(l)}(t_1)|}{|t_2-t_1|^{\alpha}}:\
t_1,t_2\in[a,b],t_1\neq t_2\Bigr\}<\infty.
$$
This space is Banach with respect to the norm
$$
\|x\|_{l,\alpha}:=\|x\|_{l}+\|x^{(l)}\|_{\alpha}'.
$$
For the sake of uniformity, we use the designations $C^{l,0}:=C^{(l)}$ and $\|\cdot\|_{l,0}:=\|\cdot\|_{l}$.

Note that each $C^{l,\alpha}$, with $0\leq\alpha\leq1$, is a Banach algebra with respect to a certain norm which is equivalent to $\|\cdot\|_{l,\alpha}$. Of course, this is true for the spaces \eqref{3.Holder} as well. The norms in them are also denoted by $\|\cdot\|_{l,\alpha}$. It will be always clear from the context to which space (scalar or vector-valued or matrix-valued functions) these norms relate.

\section{Preliminaries}\label{3sec3}

In this section, we investigate the boundary-value problem \eqref{3.syste}, \eqref{3.kue} provided that the parameter $\varepsilon$ is fixed. We establish basic properties of this problem; they will be used in our proof of Main Theorem. Fixing and then omitting $\varepsilon$ in the problem \eqref{3.syste}, \eqref{3.kue}, we write it in the form
\begin{gather}\label{3.syst}
Ly(t)\equiv  y'(t)+A(t)y(t)=f(t),\quad a\leq t\leq b,\\
By=q.\label{3.ku}
\end{gather}

So, we consider an arbitrary boundary-value problem \eqref{3.syst}, \eqref{3.ku} which is generic with respect to the function space $C^{n+1,\alpha}$. The latter means that $y\in(C^{n+1,\alpha})^m$,
$A\in(C^{n,\alpha})^{m\times m}$, $f\in(C^{n,\alpha})^m$, $B$ is a linear continuous operator from $(C^{n+1,\alpha})^m$ to $\mathbb{C}^{m}$, and $q\in\mathbb{C}^{m}$. We rewrite the problem \eqref{3.syst}, \eqref{3.ku} in the brief form $(L,B)y=(f,q)$ with the help of the continuous linear operator
\begin{equation}\label{3.L_B}
(L,B):(C^{n+1,\alpha})^{m}\rightarrow (C^{n,\alpha})^{m}\times\mathbb{C}^{m}.
\end{equation}

\begin{theorem}\label{3.t1}
The operator~\eqref{3.L_B} is Fredholm with index zero.
\end{theorem}

In connection with this theorem, we recall that a linear continuous operator $T:E_{1}\rightarrow E_{2}$ between Banach spaces $E_{1}$ and $E_{2}$ is called Fredholm if its kernel $\ker T$ and co-kernel $E_{2}/T(E_{1})$ are both
finite-dimensional. If this operator is Fredholm, then its range $T(E_{1})$ is closed in $E_{2}$ (see, e.g., \cite[Lemma~19.1.1]{Hermander85}). The finite index of the Fredholm operator $T$ is defined by the formula
$$
\mathrm{ind}\,T:=\dim\ker T-\dim(E_{2}/T(E_{1})).
$$

Before we prove Theorem~\ref{3.t1}, let us establish

\begin{lemma}\label{3.sk}
Let $A\in(C^{n,\alpha})^{m\times m}$. If a differentiable function $y:[a,b]\to\mathbb{C}^{m}$ is a solution to equation \eqref{3.syst} for a certain right-hand side $f\in(C^{n,\alpha})^{m}$, then $y\in\nobreak(C^{n+1,\alpha})^{m}$. Moreover, if $f$ runs through the whole space $(C^{n,\alpha})^{m}$, then the solutions to \eqref{3.syst} run through the whole space $(C^{n+1,\alpha})^{m}$.
\end{lemma}

\begin{proof}
We suppose that a differentiable function $y$ is a solution to equation \eqref{3.syst} for certain $f\in(C^{n,\alpha})^{m}$. Let us prove that
$y\in\nobreak(C^{n+1,\alpha})^{m}$. Since $A$ and $f$ are at least continuous on $[a,b]$, we get
$y'=f-Ay\in(C^{(0)})^{m}$. Hence, $y\in(C^{(1)})^{m}\subset(C^{0,\alpha})^{m}$.
Moreover,
\begin{equation}\label{3.impl}
\bigl(\,y\in(C^{l,\alpha})^{m}\;\Rightarrow\;
y\in(C^{l+1,\alpha})^{m}\,\bigr)
\quad\mbox{for every}\quad l\in\mathbb{Z}\cap[0,n].
\end{equation}
Indeed, if $y\in(C^{l,\alpha})^{m}$ for some integer $l\in[0,n]$, then
$y'=f-Ay\in(C^{l,\alpha})^{m}$ and therefore $y\in(C^{l+1,\alpha})^{m}$. Now the inclusion $y\in(C^{0,\alpha})^{m}$ and property \eqref{3.impl} imply the required inclusion $y\in(C^{n+1,\alpha})^{m}$.

Let us now prove the last assertion of this Lemma. For arbitrary  $f\in(C^{n,\alpha})^{m}$ there exists a solution $y$ to equation \eqref{3.syst}. We have just proved that $y\in(C^{n+1,\alpha})^{m}$. This in view of the evident implication
$$
y\in(C^{n+1,\alpha})^{m}\;\Rightarrow\;Ly\in(C^{n,\alpha})^{m}
$$
means the last assertion of Lemma~\ref{3.sk}.
\end{proof}

\begin{proof}[Proof of Theorem~$\ref{3.t1}$]
Let $Cy:=y(a)$ for arbitrary $y\in(C^{n+1,\alpha})^{m}$. The linear mapping $y\mapsto(Ly,Cy)$, with $y\in(C^{n+1,\alpha})^{m}$, sets an isomorphism
\begin{equation}\label{3.isomorphism-LC}
(L,C):(C^{n+1,\alpha})^m\leftrightarrow(C^{n,\alpha})^m\times \mathbb{C}^{m}.
\end{equation}
Indeed, for arbitrary $f\in(C^{n,\alpha})^m$ and $q\in\mathbb{C}^{m}$, the Cauchy problem \eqref{3.syst} and $y(a)=q$ has a unique solution $y$. Owing to Lemma~\ref{3.sk} we have the inclusion $y\in(C^{n+1,\alpha})^m$. Therefore the above-mentioned mapping sets the one-to-one linear operator \eqref{3.isomorphism-LC}. Since this operator is continuous, it is an isomorphism by the Banach theorem on inverse operator.

We now note that the operator~\eqref{3.L_B} is a finite-dimensional perturbation of this isomorphism. Hence, \eqref{3.L_B} is a Fredholm operator with index zero (see, e.g., \cite[Corollary~19.1.8]{Hermander85}).
\end{proof}

Let us give a criterion for the operator \eqref{3.L_B} to be invertible. We arbitrarily choose a point $t_{0}\in[a,b]$ and let $Y$ denote the matriciant of system \eqref{3.syst} relating to $t_{0}$. Thus, $Y$ is the unique solution to the matrix boundary-value problem
\begin{gather}\label{3.0.3.1}
Y'(t)=-A(t)Y(t),\quad a\leq t\leq b,\\
Y(t_{0})=I_{m}.\label{3.0.3.2}
\end{gather}
Here and below, $I_m$ is the identity $(m \times m)$-matrix. Note that $Y\in(C^{n+1,\alpha})^{m\times m}$ due to Lemma~\ref{3.sk}.

Writing $Y:=(y_{j,k})_{j,k=1}^{m}$ we let
\begin{equation}\label{3.BY}
[BY]:=\left( B\begin{pmatrix}
                                  y_{1,1} \\
                                  \vdots \\
                                  y_{m,1} \\
                   \end{pmatrix}
\ldots\;
                B\begin{pmatrix}
                                  y_{1,m} \\
                                  \vdots \\
                                  y_{m,m} \\
                   \end{pmatrix}\right).
\end{equation}
Thus, the number matrix $[BY]$ is obtained by the action of $B$ on the columns of~$Y$.

Note that
\begin{equation*}
B(Yp)=[BY]p\quad\mbox{for every}\quad p\in\mathbb{C}^{m}.
\end{equation*}
This equality is directly verified.

\begin{theorem}\label{3.t2}
The operator \eqref{3.L_B} is invertible if and only if   $\det[BY]\neq0$.
\end{theorem}

\begin{proof}
By Theorem~\ref{3.t1}, the operator~\eqref{3.L_B} is invertible if and only if its kernel is trivial. Therefore Theorem~\ref{3.t2} will be proved if we show that
\begin{equation*}
\ker(L,B)\neq\{0\}\;\Leftrightarrow\;\det[BY]=0.
\end{equation*}

Assume first that $\ker(L,B)\neq\{0\}$. Then there exists a nontrivial solution $y\in(C^{n+1,\alpha})^{m}$ to the homogeneous boundary-value problem $(L,B)y=(0,0)$. Using the matriciant $Y$ we can write this solution in the form $y=Yp$ for a certain vector $p\in\mathbb{C}^{m}\setminus\{0\}$ and get the formula
$$
0=By=B(Yp)=[BY]p.
$$
Hence, $\det[BY]=0$.

Conversely, assume that $\det[BY]=0$. Then there exists a vector $p\in\mathbb{C}^{m}\setminus\{0\}$ such that $[BY]p=0$. The function $y:=Yp\in(C^{n+1,\alpha})^{m}$ is a nontrivial solution to the homogeneous system $Ly(t)=0$ for $t\in[a,b]$. Moreover,
$$
By=B(Yp)=[BY]p=0.
$$
Hence, $0\neq y\in\ker(L,B)$.
\end{proof}

Given $t_{0}\in[a,b]$, we let $\mathcal{Y}_{t_0}^{n+1,\alpha}$ denote the set of all matrix-valued functions $Y\in(C^{n+1,\alpha})^{m\times m}$ such that $Y(t_0)=I_{m}$ and $\det Y(t)\neq0$ for every $t\in[a,b]$. This set is endowed with the metric from the Banach space $(C^{n+1,\alpha})^{m\times m}$.

\begin{theorem}\label{3.tg}
Consider the nonlinear mapping $\Upsilon:A\mapsto Y$ that associates with any $A\in(C^{n,\alpha})^{m\times m}$ the unique solution  $Y\in(C^{n+1,\alpha})^{m\times m}$ to the boundary-value problem \eqref{3.0.3.1}, \eqref{3.0.3.2}. This mapping is a homeomorphism of the Banach space $(C^{n,\alpha})^{m\times m}$ onto the metric space
$\mathcal{Y}_{t_0}^{n+1,\alpha}$.
\end{theorem}

Our proof of this theorem uses some properties of an integral operator $V_{A}$ associated with the boundary-value problem \eqref{3.0.3.1}, \eqref{3.0.3.2} with $A\in(C^{n,\alpha})^{m\times m}$.
This operator transforms any function $Y\in(C^{n+1,\alpha})^{m\times m}$ into the function
\begin{equation}\label{3.op.v}
(V_{A}Y)(t):=\int\limits_{t_0}^{t}A(s)Y(s)\,ds\quad\mbox{of}\quad t\in[a,b].
\end{equation}
Note that
\begin{equation}\label{3.equivalence}
\Upsilon A=Y\;\Leftrightarrow\;
(I+V_{A})Y=I_{m}
\end{equation}
for arbitrarily given $Y\in(C^{n+1,\alpha})^{m\times m}$. Henceforth $I$ denotes the identity operator in the relevant space.

\begin{lemma}\label{3.p.lem.1}
Let $A\in(C^{n,\alpha})^{m\times m}$. Then the linear operator $V_{A}$ is compact on $(C^{n+1,\alpha})^{m\times m}$, and its norm satisfies the inequality
\begin{equation}\label{3.int}
\|V_{A}\|\leq\varkappa\,\|A\|_{n,\alpha},
\end{equation}
where $\varkappa$ is a certain positive number that does not depend on $A$. Besides, this operator is quasinilpotent; i.e., for every $\lambda\in\mathbb{C}$ we have an isomorphism
\begin{equation}\label{3.l2.op.1}
I-\lambda V_{A}:(C^{n+1,\alpha})^{m\times m}\leftrightarrow (C^{n+1,\alpha})^{m\times m}.
\end{equation}
\end{lemma}

\begin{proof}
Given $Y\in(C^{n,\alpha})^{m\times m}$ we can write
\begin{align*}
\|V_{A}Y\|_{n+1,\alpha}&=\|V_{A}Y\|_{0}+\|AY\|_{n,\alpha}\leq
(b-a)\,\|AY\|_{0}+\|AY\|_{n,\alpha}\\
&\leq(b-a)\,c_{0}\,\|A\|_{0}\,\|Y\|_{0}+
c_{1}\,\|A\|_{n,\alpha}\,\|Y\|_{n,\alpha}.
\end{align*}
Here $c_{0}$ and $c_{1}$ are certain positive numbers that do not depend on $A$ and $Y$. Hence,
\begin{equation*}
\|V_{A}Y\|_{n+1,\alpha}\leq c_{2}\,\|A\|_{n,\alpha}\,\|Y\|_{n,\alpha}
\end{equation*}
with $c_{2}:=((b-a)c_{0}+c_{1})$. Thus, the mapping $Y\mapsto V_{A}Y$ is a bounded operator from $(C^{n,\alpha})^{m\times m}$ to $(C^{n+1,\alpha})^{m\times m}$, and its norm does not exceed $c_{2}\|A\|_{n,\alpha}$. Hence, owing to the compact embedding $(C^{n+1,\alpha})^{m\times m}\hookrightarrow(C^{n,\alpha})^{m\times m}$, we conclude that the restriction of this mapping on $(C^{n+1,\alpha})^{m\times m}$ is a compact operator on $(C^{n+1,\alpha})^{m\times m}$ and that the norm of this operator satisfies \eqref{3.int} with $\varkappa:=c_{2}c_{3}$. Here, $c_{3}$ is the norm of this compact embedding.

Let $\lambda\in\mathbb{C}$ and deduce the isomorphism \eqref{3.l2.op.1}. Since $A$ is at least continuous on $[a,b]$, the mapping $Y\mapsto Y-\lambda V_{A}Y$, with $Y\in(C^{(0)})^{m\times m}$, is an isomorphism of the space $(C^{(0)})^{m\times m}$ onto itself.
This fact is well known if $m=1$; its proof for $m\geq2$ is analogous to those for $m=1$. The restriction of this isomorphism to $(C^{0,\alpha})^{m\times m}$ is a bounded injective operator on $(C^{0,\alpha})^{m\times m}$ according to what we have proved in the previous paragraph. This operator is also surjective. Indeed, if $F\in(C^{0,\alpha})^{m\times m}$, then there exists a function $Y\in(C^{(0)})^{m\times m}$ such that $Y-\lambda V_{A}Y=F$; hence,
$$
Y=\lambda V_{A}Y+F\in(C^{0,\alpha})^{m\times m}.
$$
Thus, we have the isomorphism
\begin{equation}\label{3.l2-is-k}
I-\lambda V_{A}:(C^{k,\alpha})^{m\times m}\leftrightarrow (C^{k,\alpha})^{m\times m}.
\end{equation}
in the case of $k=0$.

We now choose an integer $l\in[0,n]$ arbitrarily and assume that the isomorphism \eqref{3.l2-is-k} holds true for $k=l$. Reasoning analogously, we can deduce from our assumption that this isomorphism holds for $k=l+1$. Namely, the restriction of the isomorphism \eqref{3.l2-is-k} for $k=l$ to $(C^{l+1,\alpha})^{m\times m}$ is a bounded injective operator on $(C^{l+1,\alpha})^{m\times m}$. Besides,
if $F\in(C^{l+1,\alpha})^{m\times m}$, then there exists a function $Y\in(C^{l,\alpha})^{m\times m}$ such that $Y-\lambda V_{A}Y=F$; hence,
$$
Y=\lambda V_{A}Y+F\in(C^{l+1,\alpha})^{m\times m}.
$$
Therefore we obtain the isomorphism \eqref{3.l2-is-k} for $k=l+1$.

Thus, we have proved the required isomorphism \eqref{3.l2.op.1} by the induction with respect to the integer $k\in[0,n]$ in \eqref{3.l2-is-k}.
\end{proof}

\begin{proof}[Proof of Theorem~$\ref{3.tg}$] If $A\in(C^{n,\alpha})^{m\times m}$, then $Y:=\Upsilon A\in\mathcal{Y}_{t_{0}}^{n+1,\alpha}$ according to Lemma~\ref{3.sk} and the Liouville-Jacobi formula and then $A(t)=-Y'(t)(Y(t))^{-1}$ for every $t\in[a,b]$. Therefore we have the injective mapping
\begin{equation}\label{3.Upsilon}
\Upsilon:(C^{n,\alpha})^{m\times m}\to\mathcal{Y}_{t_{0}}^{n+1,\alpha}.
\end{equation}
Moreover, it is surjective. Indeed, if $Y\in\mathcal{Y}_{t_{0}}^{n+1,\alpha}$, then $\Upsilon A=Y$ for
$$
A:=-Y'\,Y^{-1}\in(C^{n,\alpha})^{m\times m}.
$$

Let us show that the operator \eqref{3.Upsilon} is continuous. Assume that $A_{k}\to A$ in $(C^{n,\alpha})^{m\times m}$ as $k\to\infty$. According to Lemma~\ref{3.p.lem.1} we get the norm convergence $I+V_{A_{k}}\to I+V_{A}$, as $k\to\infty$, of the bounded operators on $(C^{n+1,\alpha})^{m\times m}$. Hence, owing to the same lemma and equivalence \eqref{3.equivalence}, we conclude that
\begin{equation*}
\Upsilon A_{k}=(I+V_{A_{k}})^{-1}I_{m}\to
(I+V_{A})^{-1}I_{m}=\Upsilon A
\end{equation*}
in $(C^{n+1,\alpha})^{m\times m}$ as $k\to\infty$. Thus, the operator \eqref{3.Upsilon} is continuous.

Its inverse is also continuous. Indeed, if $Y_{k}\to Y$ in $\mathcal{Y}_{t_{0}}^{n+1,\alpha}$ as $k\to\infty$, then
$Y_{k}'\to Y'$ in $(C^{n,\alpha})^{m\times m}$ and $Y^{-1}_{k}\to Y^{-1}$ in $(C^{n+1,\alpha})^{m\times m}$ as $k\to\infty$. Therefore
\begin{equation*}
\Upsilon^{-1}Y_{k}=-Y_{k}'\,Y^{-1}_{k}\to-Y'\,Y^{-1}=\Upsilon^{-1}Y
\end{equation*}
in $(C^{n,\alpha})^{m\times m}$ as $k\to\infty$.
\end{proof}

\section{Proof of the main results}\label{3sec4}

We will divide Main Theorem into three lemmas and prove them.

We associate the continuous linear operator
\begin{equation}\label{3.LBe}
(L(\epsilon),B(\epsilon)):(C^{n+1,\alpha})^{m}\rightarrow (C^{n,\alpha})^{m}\times\mathbb{C}^{m}
\end{equation}
with the boundary-value problem \eqref{3.syste}, \eqref{3.kue}, where $\varepsilon\in[0,\varepsilon_{1})$. According to Theorem~\ref{3.t1}, this operator is Fredholm with index zero.

Therefore Condition (0) is equivalent to that the operator \eqref{3.LBe} for $\epsilon=0$ is an isomorphism
\begin{equation}\label{3.LB0}
(L(0),B(0)):(C^{n+1,\alpha})^{m}\leftrightarrow (C^{n,\alpha})^{m}\times\mathbb{C}^{m}.
\end{equation}

\begin{lemma}\label{3.th-isomorph}
Assume that Condition \rm (0) \it and Limit Conditions \rm (I) \it and \rm (II) \it are fulfilled. Then there exists a positive number $\epsilon_{1}<\epsilon_0$ such that the operator \eqref{3.LBe} is invertible for every $\epsilon\in[0,\epsilon_1)$.
\end{lemma}

\begin{remark}\rm
Let Condition (0) be fulfilled. Then Limit Conditions (I) and (II) taken together do not imply that the operator~\eqref{3.LBe} with $0<\varepsilon\ll1$ is a small perturbation of the isomorphism \eqref{3.LB0} in the operator norm. This implication would be true if Limit Condition~(II) were replaced by the essentially stronger condition of the norm convergence of the operators $B(\epsilon)$ to $B(0)$ as $\epsilon\to0+$. Therefore the conclusion of Lemma~\ref{3.th-isomorph} does not follow from the known fact that the set of all isomorphisms between given Banach spaces is open in the uniform operator topology.
\end{remark}

\begin{proof}[Proof of Lemma $\ref{3.th-isomorph}$]
In view of Theorem~\ref{3.tg} we let
\begin{equation*}
Y(\cdot,\varepsilon):=\Upsilon A(\cdot,\varepsilon)\in\mathcal{Y}_{t_0}^{n+1,\alpha}
\quad\mbox{for each}\quad\varepsilon\in[0,\varepsilon_{0}).
\end{equation*}
Limit Condition (I) implies by this theorem that
\begin{equation}\label{3.Ylimit}
Y(\cdot,\epsilon)\to Y(\cdot,0)\quad\mbox{in}\quad (C^{n+1,\alpha})^{m\times m}\quad\mbox{as}\quad \epsilon\to0+.
\end{equation}
Hence, owing to Limit Condition (II) we obtain
\begin{equation}\label{3.BYlimit}
[B(\epsilon)Y(\cdot,\epsilon)]\to[B(0)Y(\cdot,0)]\quad\mbox{in}\quad
\mathbb{C}^{m\times m}\quad\mbox{as}\quad\epsilon\to0+.
\end{equation}
Observe that $\det[B(0)Y(\cdot,0)]\neq0$ by Condition~(0) and in view of Theorems \ref{3.t1} and~\ref{3.t2}. Therefore there exists a number $\varepsilon_{1}\in(0,\varepsilon_{0})$ such that
\begin{equation}\label{3.BYdet}
\det[B(\epsilon)Y(\cdot,\epsilon)]\neq0
\quad\mbox{for each}\quad\epsilon\in[0,\varepsilon_{1}).
\end{equation}
Thus, by Theorem~\ref{3.t2}, the operator \eqref{3.LBe} is invertible whenever $0\leq\epsilon<\varepsilon_{1}$.
\end{proof}

\begin{lemma}\label{3.th-limit}
Assume that Condition \rm (0) \it and all Limit Conditions \rm (I)--(IV) \it are fulfilled. Then the unique solution $y(\cdot,\epsilon)\in(C^{n+1,\alpha})^{m}$ to
the boundary-value problem \eqref{3.syste}, \eqref{3.kue} with $\epsilon\in(0,\epsilon_{1})$ has the limit property \eqref{3.gu}.
\end{lemma}

\begin{proof} We divide it into two steps.

\emph{Step~$1$.} Here, we prove Lemma~\ref{3.th-limit} in the case where $f(t,\varepsilon)=0$ for arbitrary $t\in[a,b]$ and $\varepsilon\in[0,\varepsilon_{0})$. We use considerations given in the proof of Lemma~\ref{3.th-isomorph}. According to this lemma, the boundary-value problem \eqref{3.syste}, \eqref{3.kue} has a unique solution $y(\cdot,\epsilon)\in(C^{n+1,\alpha})^{m}$ for each $\epsilon\in[0,\epsilon_{1})$. Since $f(\cdot,\varepsilon)\equiv0$, there exists a
vector $p(\varepsilon)\in\mathbb{C}^{m}$ such that  $y(\cdot,\varepsilon)=Y(\cdot,\varepsilon)p(\varepsilon)$. By Limit  Condition~(IV), we have the convergence
\begin{equation*}
[B(\epsilon)Y(\cdot,\epsilon)]p(\epsilon)=q(\epsilon)\to q(0)=[B(0)Y(\cdot,0)]p(0)\quad\mbox{in}\quad
\mathbb{C}^{m}\quad\mbox{as}\quad\epsilon\to0+.
\end{equation*}
Here, the equalities hold true because
\begin{equation*}
q(\epsilon)=B(\epsilon)y(\cdot,\epsilon)=
B(\epsilon)(Y(\cdot,\epsilon)p(\epsilon))=
[B(\epsilon)Y(\cdot,\epsilon)]p(\epsilon)
\end{equation*}
for each $\varepsilon\in[0,\varepsilon_{0})$. This convergence implies by \eqref{3.BYlimit} and \eqref{3.BYdet} that
\begin{equation*}
p(\epsilon)=[B(\epsilon)Y(\cdot,\epsilon)]^{-1}q(\epsilon)\to
[B(0)Y(\cdot,0)]^{-1}q(0)=p(0)
\end{equation*}
in $\mathbb{C}^{m}$ as $\epsilon\to0+$. This property and
\eqref{3.Ylimit} yield \eqref{3.gu}; namely,
\begin{equation*}
y(\cdot,\epsilon)=Y(\cdot,\varepsilon)p(\varepsilon)\to Y(\cdot,0)p(0)=y(\cdot,0)
\end{equation*}
in $(C^{n+1,\alpha})^{m}$ as $\epsilon\to0+$.

\emph{Step~$2$.} Here, we will deduce the limit property \eqref{3.gu} in the general situation from the case examined on step~1. For each $\varepsilon\in[0,\varepsilon_{1})$, we can represent the unique solution $y(\cdot,\epsilon)\in(C^{n+1,\alpha})^m$ to the boundary-value problem \eqref{3.syste}, \eqref{3.kue} in the form $y(\cdot,\epsilon)=x(\cdot,\epsilon)+\widehat{y}(\cdot,\epsilon)$. Here, $x(\cdot,\epsilon)\in(C^{n+1,\alpha})^m$ is the unique solution to the Cauchy problem
\begin{equation*}
L(\epsilon)x(t,\epsilon)=f(t,\epsilon),\quad a\leq t\leq b,\quad\mbox{and}\quad x(a,\epsilon)=0,
\end{equation*}
whereas $\widehat{y}(\cdot,\epsilon)\in(C^{n+1,\alpha})^m$ is the unique solution to the boundary-value problem
\begin{equation*}
L(\epsilon)\widehat{y}(t,\epsilon)=0,\quad a\leq t\leq b,\quad\mbox{and}\quad
B(\epsilon)\widehat{y}(\cdot,\epsilon)=\widehat{q}(\epsilon),
\end{equation*}
with
\begin{equation*}
\widehat{q}(\epsilon):=
q(\epsilon)-B(\epsilon)x(\cdot,\epsilon).
\end{equation*}
The existence and uniqueness of these three solutions is due to Lemma~\ref{3.th-isomorph}.

Let $Cx:=x(a)$ for arbitrary $x\in(C^{n+1,\alpha})^{m}$. Owing to Lemma~\ref{3.sk} and the Banach theorem on inverse operator, we have an isomorphism
\begin{equation*}
(L(\varepsilon),C):
(C^{n+1,\alpha})^m\leftrightarrow(C^{n,\alpha})^{m}\times\mathbb{C}^{m}
\end{equation*}
for each $\varepsilon\in[0,\varepsilon_{0})$. According to Limit Condition~(I) we obtain the norm convergence $L(\varepsilon)\to L(0)$, as $\varepsilon\to0+$, of the bounded operators from $(C^{n+1,\alpha})^m$ to $(C^{n,\alpha})^{m}$. Indeed, for arbitrary $\varepsilon\in[0,\varepsilon_{0})$ and $z\in(C^{n+1,\alpha})^{m}$, we can write
\begin{align*}
\|(L(\varepsilon)-L(0))z\|_{n,\alpha}&=
\|(A(\cdot,\varepsilon)-A(\cdot,0))z\|_{n,\alpha}\\
&\leq c_{0}\,\|A(\cdot,\varepsilon)-A(\cdot,0)\|_{n,\alpha}\,\|z\|_{n,\alpha}\\
&\leq c_{0}\,c_{1}\,
\|A(\cdot,\varepsilon)-A(\cdot,0)\|_{n,\alpha}\,\|z\|_{n+1,\alpha};
\end{align*}
here, $c_{0}$ and $c_{1}$ are certain positive numbers that do not depend on $\varepsilon$ and $z$. The latter convergence implies the norm convergence $(L(\varepsilon),C)^{-1}\to(L(0),C)^{-1}$, as $\varepsilon\to0+$, of the bounded operators from $(C^{n,\alpha})^{m}\times\mathbb{C}^{m}$ to $(C^{n+1,\alpha})^{m}$. Hence, we conclude by Limit Condition~(III) that
\begin{equation}\label{3.conv-x}
\begin{gathered}
x(\cdot,\varepsilon)=(L(\varepsilon),C)^{-1}(f(\cdot,\varepsilon),0)\to
(L(0),C)^{-1}(f(\cdot,0),0)=x(\cdot,0)\\
\mbox{in}\quad(C^{n+1,\alpha})^{m}\quad\mbox{as}\quad\epsilon\to0+.
\end{gathered}
\end{equation}

It follows from this convergence and Limit Conditions (II) and (IV) that
\begin{equation*}
\widehat{q}(\epsilon)=q(\epsilon)-B(\epsilon)x(\cdot,\epsilon)\to
q(0)-B(0)x(\cdot,0)=\widehat{q}(0)
\end{equation*}
in $\mathbb{C}^{m}$ as $\epsilon\to0+$. Hence, according to step~1, Limit Conditions (I), (II) and (IV) imply the convergence
\begin{equation}\label{3.conv-y-hat}
\widehat{y}(\cdot,\epsilon)\to \widehat{y}(\cdot,0)\quad\mbox{in}\quad(C^{n+1,\alpha})^{m}
\quad\mbox{as}\quad\epsilon\to0+.
\end{equation}
Now the relations \eqref{3.conv-x} and \eqref{3.conv-y-hat} yield the  required property \eqref{3.gu} for $y(\cdot,\epsilon)=x(\cdot,\epsilon)+\widehat{y}(\cdot,\epsilon)$.
\end{proof}

\begin{lemma}\label{3.th-necessity}
Assume that  the boundary-value problem \eqref{3.syste}, \eqref{3.kue} satisfies Basic Definition. Then this problem satisfies Limit Conditions \rm (I) \it and \rm (II)\it.
\end{lemma}

\begin{proof}
We divide it into three steps.

\emph{Step~$1$.} We will prove here that the boundary-value problem \eqref{3.syste}, \eqref{3.kue} satisfies Limit Condition~(I). According to condition ($\ast$) of Basic Definition, the operator \eqref{3.LBe} is invertible for each $\varepsilon\in[0,\varepsilon_{1})$. Given $\varepsilon\in[0,\varepsilon_{1})$, we consider the matrix boundary-value problem
\begin{gather}\label{3.matrix-eq}
Y'(t,\varepsilon)+A(t,\varepsilon)Y(t,\varepsilon)=0\cdot I_{m},
\quad a\leq t\leq b,\\
[BY(\cdot,\varepsilon)]=I_{m}. \label{3.matrix-bound-cond}
\end{gather}
Note that it is a union of $m$ boundary-value problem \eqref{3.syste}, \eqref{3.kue} whose right-hand sides are independent of $\varepsilon$.
Therefore the matrix problem \eqref{3.matrix-eq}, \eqref{3.matrix-bound-cond} has a unique solution $Y(\cdot,\varepsilon)\in(C^{n+1,\alpha})^{m\times m}$. Moreover, it follows from condition $(\ast\ast)$ of Basic Definition that
\begin{equation}\label{3.lim-Y}
Y(\cdot,\varepsilon)\to Y(\cdot,0)\quad\mbox{in}\quad (C^{n+1,\alpha})^{m\times m}\quad\mbox{as}\quad\varepsilon\to0+.
\end{equation}
We assert that
\begin{equation}\label{3.det-not0}
\det Y(t,\varepsilon)\neq0\quad\mbox{for each}\quad t\in[a,b].
\end{equation}
Indeed, if \eqref{3.det-not0} were wrong, the function columns of the matrix $Y(\cdot,\varepsilon)$ would be linearly dependent, which would contradict \eqref{3.matrix-bound-cond}. Now, formulas \eqref{3.lim-Y} and \eqref{3.det-not0} yield the required convergence
\begin{equation*}
A(\cdot,\varepsilon)=-Y'(\cdot,\varepsilon)(Y(\cdot,\varepsilon))^{-1}\to
-Y'(\cdot,0)(Y(\cdot,0))^{-1}=A(\cdot,0)
\end{equation*}
in $(C^{n,\alpha})^{m\times m}$ as $\varepsilon\to0+$. So, the boundary-value problem \eqref{3.syste}, \eqref{3.kue} satisfies Limit Condition~(I). Note that then
\begin{equation}\label{3.bound-norm-A}
\sup\bigl\{\|A(\epsilon)\|_{n,\alpha}:
\varepsilon\in[0,\epsilon_{3})\bigr\}<\infty\quad\mbox{for some number}\quad\varepsilon_{3}\in(0,\varepsilon_{1}).
\end{equation}

\emph{Step~$2$.} To prove that this problem satisfies Limit Condition~(II), we will show on this step that
\begin{equation}\label{3.bound-norm-B}
\sup\bigl\{\|B(\epsilon)\|:
\varepsilon\in[0,\epsilon_{4})\bigr\}<\infty\quad\mbox{for some number}\quad\varepsilon_{4}\in(0,\varepsilon_{1}).
\end{equation}
Here, $\|\cdot\|$ stands for the norm of the bounded operator
\begin{equation*}
B(\epsilon):(C^{n+1,\alpha})^m\to\mathbb{C}^{m}.
\end{equation*}
Suppose the contrary; then there exists a sequence of numbers $\varepsilon^{(k)}\in(0,\varepsilon_{1})$, with $1\leq k\in\mathbb{Z}$, such that $\varepsilon^{(k)}\to0$ and
\begin{equation}\label{3.norm-B-to-infty}
0<\|B(\epsilon^{(k)})\|\to\infty\quad\mbox{as}\quad k\to\infty.
\end{equation}
Given an integer $k\geq1$, we can choose a function $z_{k}\in(C^{n+1,\alpha})^{m}$ such that
\begin{equation}\label{3.property-z}
\|z_{k}\|_{n+1,\alpha}=1\quad\mbox{and}\quad
\|B(\epsilon^{(k)})z_{k}\|_{\mathbb{C}^{m}}\geq
\frac{1}{2}\,\|B(\epsilon^{(k)})\|.
\end{equation}
We let
\begin{gather*}
y(\cdot,\epsilon^{(k)}):=
\|B(\epsilon^{(k)})\|^{-1}\,z_{k}\in(C^{n+1,\alpha})^{m},\\
f(\cdot,\epsilon^{(k)}):=L(\epsilon^{(k)})\,y(\cdot,\epsilon^{(k)})
\in(C^{n,\alpha})^{m},\\
q(\epsilon^{(k)}):=B(\epsilon^{(k)})\,y(\cdot,\epsilon^{(k)})\in
\mathbb{C}^{m}.
\end{gather*}
It follows from \eqref{3.norm-B-to-infty} and \eqref{3.property-z} that
\begin{equation}\label{3.limit-y}
y(\cdot,\epsilon^{(k)})\to0\quad\mbox{in}\quad(C^{n+1,\alpha})^{m}\quad
\mbox{as}\quad k\to\infty.
\end{equation}
This implies that
\begin{equation}\label{3.limit-f}
f(\cdot,\epsilon^{(k)})\to0\quad\mbox{in}\quad(C^{n,\alpha})^{m}\quad
\mbox{as}\quad k\to\infty
\end{equation}
because $A(\cdot,\varepsilon)$ satisfies Limit Condition~(I) as we have proved on step~1. Besides, it follows from \eqref{3.property-z} that
\begin{equation*}
\frac{1}{2}\leq\|q(\epsilon^{(k)})\|_{\mathbb{C}^{m}}\leq1.
\end{equation*}
Hence, there is a convergent subsequence $(q(\epsilon^{(k_j)}))_{j=1}^{\infty}$ of the sequence $(q(\epsilon^{(k)}))_{k=1}^{\infty}$, with
\begin{equation}\label{3.limit-q}
q(0):=\lim_{j\to\infty}q(\epsilon^{(k_j)})\neq0
\quad\mbox{in}\quad\mathbb{C}^{m}.
\end{equation}

Thus, for every integer $j\geq1$, the function $y(\cdot,\epsilon^{(k_j)})\in(C^{n+1,\alpha})^{m}$ is a unique solution to the boundary-value problem
\begin{gather*}
L(\epsilon^{(k_j)})\,y(t,\epsilon^{(k_j)})=f(t,\epsilon^{(k_j)}),
\quad a\leq t\leq b,\\
B(\epsilon^{(k_j)})\,y(\cdot,\epsilon^{(k_j)})=q(\epsilon^{(k_j)}).
\end{gather*}
According to condition $(\ast\ast)$ of Basic Definition, it follows from \eqref{3.limit-f} and \eqref{3.limit-q} that the function $y(\cdot,\epsilon^{(k_j)})$ converges to the unique solution $y(\cdot,0)$ of the limiting boundary-value problem
\begin{equation*}
L(0)y(t,0)=0,\quad a\leq t\leq b,\qquad\mbox{and}\qquad
B(0)y(\cdot,0)=q(0).
\end{equation*}
This convergence holds in $(C^{n+1,\alpha})^{m}$ as $t\to0+$. Owing to \eqref{3.limit-y} we conclude that $y(\cdot,0)=0$, which contradicts the boundary condition $B(0)y(\cdot,0)=q(0)$ with $q(0)\neq0$. So, our assumption is wrong, and we have thereby proved~\eqref{3.bound-norm-B}.

\emph{Step~$3$.} Let us prove that the boundary-value problem \eqref{3.syste}, \eqref{3.kue} satisfies Limit Condition~(II). Owing to \eqref{3.bound-norm-A} and \eqref{3.bound-norm-B} we have
\begin{equation}\label{3.bound-norm-LB}
\varkappa':=\sup\bigl\{\|(L(\epsilon),B(\epsilon))\|:
\varepsilon\in[0,\epsilon_{2}')\bigr\}<\infty,
\end{equation}
with $\varepsilon_{2}':=\min\{\varepsilon_{3},\varepsilon_{4}\}<
\varepsilon_{1}$ and with $\|\cdot\|$ denoting the norm of the operator \eqref{3.LBe}. Choose a function $y\in(C^{n+1,\alpha})^{m}$ arbitrarily and put $f(\cdot,\varepsilon):=L(\varepsilon)y\in(C^{n,\alpha})^{m}$ and $q(\varepsilon):=B(\varepsilon)y\in\mathbb{C}^{m}$ for each $\varepsilon\in[0,\varepsilon_{0})$. Thus,
\begin{equation}\label{3.y-as-preimage}
y=(L(\varepsilon),B(\varepsilon))^{-1}
(f(\cdot,\varepsilon),q(\varepsilon))\quad\mbox{for every}\quad
\varepsilon\in[0,\varepsilon_{1}).
\end{equation}
Here, of course, $(L(\varepsilon),B(\varepsilon))^{-1}$ stands for the bounded operator
\begin{equation}\label{3.inverse}
(L(\epsilon),B(\epsilon))^{-1}:(C^{n,\alpha})^{m}\times\mathbb{C}^{m}
\to(C^{n+1,\alpha})^{m}.
\end{equation}
According to condition $(\ast\ast)$ of Basic Definition, we conclude that
\begin{equation}\label{3.convergence-preimage}
\begin{gathered}
(L(\varepsilon),B(\varepsilon))^{-1}(f,q)\to(L(0),B(0))^{-1}(f,q)
\;\;\mbox{in}\;\;(C^{n+1,\alpha})^{m}\;\;\mbox{as}\;\;\varepsilon\to0+\\
\mbox{for arbitrary}\;\;f\in(C^{n,\alpha})^{m}\;\;\mbox{and}\;\;
q\in\mathbb{C}^{m}.
\end{gathered}
\end{equation}
Making use of \eqref{3.bound-norm-LB}, \eqref{3.y-as-preimage}, and \eqref{3.convergence-preimage} successively, we now obtain the following relations for every $\varepsilon\in[0,\varepsilon_{2}')$:
\begin{align*}
&\bigl\|B(\varepsilon)y-B(0)y\bigr\|_{\mathbb{C}^{m}}\leq
\bigl\|(f(\cdot,\varepsilon),q(\varepsilon))-
(f(\cdot,0),q(0))\bigr\|_{(C^{n,\alpha})^{m}\times\mathbb{C}^{m}}\\
&=\bigl\|(L(\varepsilon),B(\varepsilon))
(L(\varepsilon),B(\varepsilon))^{-1}
\bigl((f(\cdot,\varepsilon),q(\varepsilon))-
(f(\cdot,0),q(0))\bigr)\bigr\|_{(C^{n,\alpha})^{m}\times\mathbb{C}^{m}}\\
&\leq\varkappa'\,\bigl\|(L(\varepsilon),B(\varepsilon))^{-1}
\bigl((f(\cdot,\varepsilon),q(\varepsilon))-
(f(\cdot,0),q(0))\bigr)\bigr\|_{n+1,\alpha}\\
&=\varkappa'\,\bigl\|(L(0),B(0))^{-1}(f(\cdot,0),q(0))-
(L(\varepsilon),B(\varepsilon))^{-1}(f(\cdot,0),q(0))\bigr\|_{n+1,\alpha}
\to0
\end{align*}
as $\varepsilon\to0+$. Since $y\in(C^{n+1,\alpha})^{m}$ is arbitrarily chosen, we have proved that the boundary-value problem \eqref{3.syste}, \eqref{3.kue} satisfies Limit Condition~(II).
\end{proof}

We can see that Lemmas \ref{3.th-isomorph}, \ref{3.th-limit}, and \ref{3.th-necessity} constitute Main Theorem.

\begin{proof}[Proof of Theorem~$\ref{3.th-bound}$]
Let us first prove the left-hand part of the two-sided estimate~\eqref{3.bound}. According to Limit Conditions (I) and (II) the bounded operator \eqref{3.LBe} converges strongly to the bounded operator \eqref{3.LB0} as $\epsilon\to0+$. Hence, there exists a positive number $\epsilon_{2}'<\epsilon_{1}$ such that \eqref{3.bound-norm-LB} holds true, with $\|\cdot\|$ denoting the norm of the operator \eqref{3.LBe}. Indeed, supposing the contrary, we obtain a sequence of positive numbers $\varepsilon^{(k)}$, with $1\leq k\in\mathbb{Z}$, such that $\varepsilon^{(k)}\to0$ and $\|(L(\epsilon^{(k)}),B(\epsilon^{(k)})\|\to\infty$ as $k\to\infty$. In view of the Banach-Steinhaus theorem, this contradicts the strong convergence of $(L(\epsilon^{(k)}),B(\epsilon^{(k)}))$ to $((L(0),B(0))$ as $k\to\infty$. So, for every $\epsilon\in(0,\epsilon_{2}')$ we have the bound
\begin{gather*}
\|L(\epsilon)y(\cdot,0)-f(\cdot,\epsilon))\|_{n,\alpha}+
\|B(\epsilon)y(\cdot,0)-q(\epsilon)\|_{\mathbb{C}^{m}}\leq
\varkappa'\|y(\cdot,0)-y(\cdot,\epsilon)\|_{n+1,\alpha};
\end{gather*}
i.e., we obtain the left-hand part of the two-sided estimate \eqref{3.bound} with $\varkappa_{1}:=1/\varkappa'$.

Let us now prove the right-hand part of this estimate. According to Lemma~\ref{3.th-isomorph}, the operator \eqref{3.LBe} has a bounded inverse operator \eqref{3.inverse} for every $\epsilon\in[0,\epsilon_1)$. Moreover, $(L(\epsilon),B(\epsilon))^{-1}$ converges strongly to $(L(0),B(0))^{-1}$ as $\varepsilon\to0+$. Indeed, choosing $f\in(C^{n,\alpha})^{m}$ and $q\in\mathbb{C}^{m}$ arbitrarily, we conclude by Lemma~\ref{3.th-limit} that
\begin{equation*}
(L(\epsilon),B(\epsilon))^{-1}(f,q)=:y(\cdot,\varepsilon)\to
y(\cdot,0):=(L(0),B(0))^{-1}(f,q)
\end{equation*}
in $(C^{n+1,\alpha})^{m}$ as $\varepsilon\to0+$. Hence, there exists a positive number $\epsilon_{2}<\epsilon_{2}'$ such that
\begin{equation*}
\varkappa_{2}:=\sup\bigl\{\|(L(\epsilon),B(\epsilon))^{-1}\|:
\varepsilon\in[0,\epsilon_{2})\bigr\}<\infty,
\end{equation*}
with $\|\cdot\|$ denoting the norm of the operator \eqref{3.inverse}. This follows from the Banach-Steinhaus theorem in a way analogous to that used in the previous paragraph. So, for every $\epsilon\in(0,\epsilon_{2})$ we have the bound
\begin{align*}
\|&y(\cdot,0)-y(\cdot,\epsilon)\|_{n+1,\alpha}\\
&\leq \varkappa_{2}
\bigl(\,\|L(\epsilon)(y(\cdot,0)-y(\cdot,\epsilon))\|_{n,\alpha}+
\|B(\epsilon)(y(\cdot,0)-y(\cdot,\epsilon))\|_{\mathbb{C}^{m}}\bigr),
\end{align*}
i.e. the right-hand part of the the two-sided estimate \eqref{3.bound}.
\end{proof}

\end{document}